\documentclass[12pt]{amsart}
\usepackage{amsmath}
\usepackage{amsthm}
\usepackage{amsbsy}
\usepackage{amssymb}
\usepackage{latexsym}
\usepackage[capitalize]{cleveref}

\newcommand \Appr{{\rm{ Appr}}}
\newcommand \NN{\mathbb{N}}
\newcommand \RR{\mathbb{R}}

\def\bbN{{\mathbb N}}

\newtheorem{thm}{Theorem}[section]
\newtheorem{lem}[thm]{Lemma}
\newtheorem{dfn}[thm]{Definition}

\newtheorem{pro}[thm]{Proposition}

\newtheorem{lemma}[thm]{Lemma}

\newtheorem{prob}[thm]{Problem}
\newtheorem{rem}[thm]{Remark}
\newtheorem{question}[thm]{Question}

\author[A.S. Kechris]{Alexander S. Kechris}
\address{Department of Mathematics, Caltech, 1200 E. California Blvd,
Pasadena, CA 91125, USA}
\email{kechris@caltech.edu}
\urladdr{http://www.math.caltech.edu/~kechris/}

\author[M. Malicki]{Maciej Malicki}
\address{Institute of Mathematics, Polish Academy of Sciences, ul. Sniadeckich 8, 00-656, Warsaw, Poland}
\email{mamalicki@gmail.com}

\author[A. Panagiotopoulos]{Aristotelis Panagiotopoulos}
\address{Institut f\"ur Mathematische Logik und Grundlagenforschung, Westfalische Wilhelms-Universit\"at M\"unster, Einsteinstr. 62, 48149 M\"unster, Germany}
\email{apanagio@uni-muenster.de}

\author[J. Zielinski]{Joseph Zielinski}
\address{Department of Mathematics, GAB 435, University of North Texas,
Denton, TX 76201, USA}
\email{Joseph.Zielinski@unt.edu}
\email{zielinski.math@gmail.com}



\thanks{Research of A.S. Kechris  was partially supported by NSF Grants DMS-1464475 and  DMS-1950475 }

\begin{document}
\title[Polish groups with non-essentially countable actions]{On Polish groups admitting non-essentially countable actions}

\begin{abstract}
It is a long-standing open question whether every Polish group that is not locally compact admits a Borel action on a standard Borel space whose associated orbit equivalence relation is not essentially countable. We answer this question positively for the class of all Polish groups that embed in the isometry group of a locally compact metric space. This class contains all non-archimedean Polish groups, for which we provide an alternative proof based on a new criterion for non-essential countability. Finally, we  provide the following variant of a theorem of Solecki: every infinite-dimensional Banach space has a continuous action whose orbit equivalence relation is Borel but not essentially countable.
\end{abstract}

\maketitle

\section{Introduction}

Motivated by foundational questions about the intrinsic complexity of various mathematical classification problems, one of the prominent ongoing projects in descriptive set theory seeks to organize the collection of all definable equivalence relations with respect to Borel reductions.

\begin{dfn}
Let $E$ and $F$ be analytic equivalence relations on standard Borel spaces $X$ and $Y$ respectively. We say that $E$ {\bf Borel reduces} to $F$, in symbols $E\leq_B F$, if there is a Borel map $f\colon X\to Y$ so that \[x E x' \iff f(x) F f(x').\]
\end{dfn}

The simplest equivalence relations in the Borel reduction complexity hierarchy are the ones that can be classified up to the equality relation $=_Y$ on some Polish space $Y$: an equivalence relation $E$ on $X$ is {\bf concretely classifiable} (or {\bf smooth})  if there is a Polish space $Y$ so that $E\leq_B =_Y$. Another important  class of equivalence relations of relatively low complexity is the class of \emph{essentially countable} equivalence relations. A Borel equivalence relation $E$ on $X$ is {\bf countable} if every $E$-class is countable.  We say that $E$ on $X$ is {\bf essentially countable} if $E\leq_B F$, for some countable Borel equivalence relation $F$. 

In many interesting cases the equivalence relation under consideration is induced by a Borel action of a Polish group on a standard Borel space. In fact, by the Feldman-Moore theorem \cite{Feldman} every countable Borel equivalence relation is the \emph{orbit equivalence relation} of a countable (discrete) group.
Let $G$ be a Polish group. A {\bf Borel $G$-space} is a standard Borel space $X$ together with a Borel action of $G$ on $X$. If $X$ is additionally Polish and the action  continuous,  then we call the Borel $G$-space $X$ a {\bf Polish $G$-space}. If $X$ is a Borel $G$-space then by \cite[Theorem 5.2.1]{BK}  we can always replace $X$ with a Polish space that is Borel isomorphic to it so that the new space is a Polish $G$-space.
 If $x\in X$, we denote by $[x]$ the $G$-orbit of $x$. To each Borel $G$-space $X$ we associate the {\bf orbit equivalence relation} $E^G_X$ defined by setting $x E^G_X x'$ if and only if $[x]=[x']$.

Often, dynamical properties of a  $G$-space $X$ are reflected in the complexity of the associated  orbit equivalence relation $E^G_X$. As a consequence, topological restrictions on $G$ impose restrictions on the complexity of $E^G_X$. For example, every orbit equivalence relation induced by a Borel action of a compact Polish group is concretely classifiable since the assignment $x\mapsto [x]$ is a Borel reduction from $E^G_X$ to the Polish space of all compact subsets of $X$. In \cite{So}, Solecki  provides a converse to this fact: if $G$ is not compact there is a Borel $G$-space whose orbit equivalence relation is not concretely classifiable. 

Similarly to the compact case, it is a theorem of Kechris \cite{K92},  that every orbit equivalence relation induced by any locally compact Polish group $G$ is essentially countable. The question of whether this theorem too admits a converse was raised in \cite{K92}; see also \cite[Problem 3.15]{K19}:

\begin{prob}\label{Problem}
Let $G$ be a Polish group with the property that every equivalence relation induced by a Borel action of $G$ on standard Borel space is {\bf essentially countable}. Is $G$ locally compact?
\end{prob}

In \cite{Hj00}, G. Hjorth referred to this problem as ``stubbornly open,'' and it so remains up to this day. However, an affirmative answer has been obtained so far for certain classes of Polish groups:
\begin{enumerate}
\item[(i)]  \cite{Tho} All Polish groups that do not admit a complete left-invariant metric;
\item[(ii)] \cite{So} All separable Banach spaces, viewed as groups under addition;
\item[(iii)] \cite{Ma} All abelian isometry groups of separable locally compact metric spaces.
\end{enumerate}

In this paper we establish further progress regarding Problem \ref{Problem}, in several directions. First we provide the following generalization of (iii) above. 

\begin{thm}\label{T:1}
Let $G$ be the group of isometries of a separable locally compact metric space. If all orbit equivalence relations induced by Borel actions of $G$ on standard Borel spaces are essentially countable, then $G$ is locally compact.
\end{thm}
It has been shown in \cite[Theorem 6.3]{GK} that every closed subgroup of the group of isometries of a separable locally compact metric space is, up to topological group isomorphism, also the isometry group of a separable locally compact metric space. So this result also applies to closed subgroups of the group of isometries of a separable locally compact metric space. In particular,
the class of groups described in the statement of Theorem \ref{T:1} contains all non-archimedean Polish groups, i.e., all groups of the form $\mathrm{Aut}(\mathcal{N})$, where $\mathcal{N}$ is a  countable structure.

In fact we show (in \cref{23}) that \cref{T:1} holds for all spatial closed subgroups of the automorphism group of a standard, non-atomic probability space, a class of groups that contains all isometry groups of separable locally compact metric spaces. 

In \cite{So}, Solecki shows that every infinite dimensional separable Banach space admits a Borel action on a standard Borel space whose associated orbit equivalence relation is analytic non-Borel. Since every essentially countable equivalence relation is Borel, one may conclude with point (ii) above. Our next result shows that the  non-essentially countable dynamics of  such a Banach space can be witnessed by an orbit equivalence relation which, additionally, is Borel. 

\begin{thm}\label{T:2}
Let $Z$ be an infinite dimensional separable Banach space viewed as a group under addition. Then there exists a Polish $Z$-space $X$ whose associated orbit equivalence relation is Borel and not essentially countable.
\end{thm}

While the proof of Theorem \ref{T:1} involves techniques from  measurable dynamics, the proof of Theorem \ref{T:2} relies only on Hjorth's notion of \emph{storminess} \cite{Hj05} and it is purely topological.  It is natural to ask whether a purely topological argument exists for Theorem \ref{T:1}. In Section \ref{S:Bernoulli} we provide such an argument for the subclass of all non-archimedean Polish groups. In fact this argument uses a new criterion for showing that an orbit equivalence relation is not essentially countable, developed in 
 Section \ref{S:Game}. As in the case of Hjorth's notion of storminess, this criterion involves only the topological aspects of the action. Moreover, 
its proof is elementary and it is based  on a lemma used in the theory of turbulence \cite{Hj00} as well as in \cite{LP18}.  We apply this criterion to a \emph{Bernoulli shift} type of action that can be associated to every non-archimedean Polish group. We show how these Bernoulli shift actions reflect many other topological properties of such groups.
 
\subsubsection*{Acknowledgments}

We are grateful to  Anush Tserunyan, Forte Shinko and Todor Tsankov
for their active interest in this project and their helpful suggestions. We also thank the anonymous referee for many valuable detailed comments and for providing alternative arguments for some results below.

\section{Proof of \cref{T:1}}\label{S1}
{\bf (A)} We will show that every non-locally compact closed subgroup of the isometry group of a separable locally compact metric space admits a Borel action on a standard Borel space whose associated equivalence relation is not essentially countable.

The proof will make essential use of the following two results. The first gives a sufficient condition for a non-locally compact Polish group to have a Borel action on a standard Borel space with associated equivalence relation not essentially countable. It is a special case of \cite[Theorem A]{FR} (see also \cite[Theorem 1.6]{K92}):

\begin{thm}[{\cite[Theorem A]{FR}}]\label{21}
Let $G$ be a non-locally compact Polish group which has a free Borel action on a standard Borel space $X$ which admits an invariant probability Borel measure. Then $E_X^G$ is not essentially countable.
\end{thm}
Recall that an action of a group $G$ on a set $X$ is {\bf free} if for every $g\not=1$ and every $x$, $g\cdot x \not= x$. Also for any action of $G$ on $X$ we denote by $F(X)$ its {\bf free part} defined by:
\[
x \in F(X) \iff \forall g\not= 1(g\cdot x \not= x).
\]
This is a $G$-invariant subset of $X$ on which $G$ acts freely.

The second result is a structure theorem for isometry groups of separable locally compact metric spaces. Below $S_\infty$ is the Polish group of all permutations of $\bbN$. For any group $G$, $S_\infty$ acts on the product group $G^\bbN$ by $\sigma \cdot (g_n) = (g_{\sigma^{-1}(n)})$ and we denote by $S_\infty\ltimes G^\bbN$ the semidirect product with multiplication given by $(\sigma, (g_n))(\tau, (h_n)) = (\sigma\tau, (g_n) (\sigma\cdot (h_n))$.

\begin{thm}[{\cite[Chapter 6]{GK}}]
Up to topological group isomorphism, the isometry groups of separable locally compact metric spaces are exactly the closed subgroups of groups of the form
\[
\prod_n(S_\infty \ltimes K_n^\bbN),
\]
where each $K_n$ is a Polish locally compact group.
\end{thm}

\medskip
{\bf (B)} Let now $(X,\mu)$ be standard non-atomic probability space and denote by ${\rm Aut}(X,\mu)$ the Polish group of all automorphisms of $(X,\mu)$ with the weak topology. If $G$ is a closed subgroup of ${\rm Aut}(X,\mu)$, then the identity gives a Boolean action of $G$ on $(X,\mu)$. We say that $G$ is {\bf spatial} if this action has a spatial realization, i.e., it comes from a Borel action of $G$ on $X$ which preserves $\mu$. Note that not every closed subgroup of  ${\rm Aut}(X,\mu)$, including   ${\rm Aut}(X,\mu)$ itself, is spatial.
For more information about Boolean actions and spatial realizations, see \cite{GTW}, \cite{GW}.

We now have the following  result:
\begin{thm}\label{23}
Let $(X,\mu)$ be standard non-atomic probability space and $G$ a spatial closed subgroup of ${\rm Aut}(X,\mu)$ and denote by $g\cdot x$ the corresponding action of $G$ on $X$. Consider the diagonal action of $G$ on $X^\bbN$ given by $g\cdot (x_n) = (g\cdot x_n)$, which preserves the product measure $\mu^\bbN$. Then there is a $G$-invariant Borel set $Y\subseteq X^\bbN$with $\mu^\bbN(Y)=1$ such that $G$ acts freely on $Y$. In particular, if $G$ is non-locally compact, $E_Y^G$ is not essentially countable.
\end{thm}

\begin{proof} 
By \cite[Theorem 5.2.1]{BK}, we can assume that $X$ is Polish and the action of $G$ on $X$ is continuous. For $g\in G\setminus \{1\}$, let 
\[
A_g = \{ x\in X \colon g\cdot x \not= x\},
\]
so that $\mu (A_g) > 0$. If $x\in A_g$, there is an open nbhd $N_g^x$ of $g$ with $1\notin N_g^x$ and an open nbhd $V_g^x$ of $x$ such that
\[
N_g^x\cdot V_g^x \cap V_g^x =\emptyset.
\]
So
 \[
A_g \subseteq \bigcup \{V_g^x\colon x\in A_g\},
\]
and thus there is a sequence $x_0, x_1, \dots \in A_g$ such that
\[
A_g \subseteq V_g^{x_0} \cup V_g^{x_1} \cup \cdots.
\]
Therefore there is an $n$ such that $\mu (V_g^{x_n}) >0$.

So we have shown that if $g\not= 1$, then there is an open nbhd $N_g$ of $g$ with $1\notin N_g$ and an open set $V_g$ such that $\mu (V_g) >0$ and $N_g\cdot V_g \cap V_g =\emptyset$. So $G\setminus\{1\} = \bigcup_{g\not= 1} N_g$ and therefore there is a sequence $g_0, g_1, \dots$ with $g_n\not= 1, \forall n$, and $G\setminus\{1\} = N_{g_0} \cup N_{g_1}\cdots$.

Consider now the action of $G$ on $(X^\bbN, \mu^\bbN)$ and its free part $F(X^\bbN)$, which is a co-analytic, therefore measurable set. We will check that it has $\mu^\bbN$-measure 1. Denote by $Z$ its complement, i.e., 
\[
Z = \{z\in X^\bbN \colon \exists g\not= 1 (g\cdot z = z)\}.
\]
We want to show that $\mu^\bbN(Z) =0$. Since $Z = \bigcup_k Z_k$, where
\[
Z_k = \{z\in X^\bbN \colon \exists g\in N_{g_k} (g\cdot z = z)\}
\]
it is enough to show that $\mu^\bbN(Z_k) =0, \forall k$. But if $g\in N_{g_k}$ and $g\cdot z = z$, then $z_n \notin V_{g_k}, \forall n$. i.e.,
\[
Z_k \subseteq (X\setminus V_{g_k})^\bbN,
\]
and, since $\mu (X\setminus V_{g_k}) <1$, $\mu^\bbN (Z_k) = 0$.

Finally we will find a Borel $G$-invariant set $Y\subseteq F(X^\bbN)$ with $\mu^\bbN (Y) =1$. For each $A\subseteq X^\bbN$, let $A^* = G\cdot A$ be its $G$-saturation. Define then inductively $Y_n\subseteq F(X^\bbN)$ as follows: $Y_0$ is a Borel subset of $F(X^\bbN)$ with $\mu^\bbN (Y_0) =1$; $Y_{n+1}$ is a Borel subset of $F(X^\bbN)$ containing $Y_n^*$; this exists by the Separation Theorem for analytic sets. Then $Y= \bigcup_n Y_n$ works.
\end{proof}

\medskip
{\bf (C)} To complete the proof of \cref{T:1}, it is thus enough to show that every isometry group of a separable locally compact metric space is (up to topological group isomorphism) a spatial closed subgroup of ${\rm Aut}(X,\mu)$. This follows from the next two results:

\begin{pro}
For each sequence $(K_n)$ of Polish locally compact groups, the group $\prod_n(S_\infty \ltimes K_n^\bbN)$ is (up to topological group isomorphism) a closed subgroup of ${\rm Aut}(X,\mu)$.
\end{pro}
\begin{proof}
Let $H$ be the infinite dimensional separable (complex) Hilbert space and $U(H)$ its unitary group (with the strong topology). By \cite[Appendix E]{K11}, $U(H)$ is a closed subgroup of ${\rm Aut}(X,\mu)$, so it is enough to show that the group $\prod_n(S_\infty \ltimes K_n^\bbN)$ is a closed subgroup of $U(H)$.

First note that if $K$ is a Polish locally compact group with Haar measure $\eta$, then by identifying each element of $K$ with the associated left-translation action on $L^2(K, \eta)$ (the regular representation), we see that $K$ is a closed subgroup of $U(H)$.

Next observe that $U(H)^\bbN$ is also a closed subgroup of $U(H)$. To see this, write $H = \bigoplus_n H_n$, where each $H_n$ is a copy of $H$
and note that $U(H)^\bbN$ can be identified with the closed subgroup of $U(H)$ consisting of all $T\in U(H)$ that satisfy $T(H_n) = H_n$, for all $n\in\mathbb{N}$. 

Finally, we show that $S_\infty \ltimes U(H)^\bbN$ is a closed subgroup of $U(H)$. Represent again $H$ as $H = \bigoplus_n H_n$ as above. Then identify any $g\in S_\infty$ with $T_g\in U(H)$, where $T_g$ sends $H_n$ to $H_{g(n)}$ by the identity. Call $\widehat{S}_\infty$ this copy of $S_\infty$ within the group $U(H)$. Then $S_\infty \ltimes U(H)^\bbN$ can be identified with the internal product $ \widehat{S}_\infty U(H)_\bbN$, which is a closed subgroup of $U(H)$.
\end{proof}

\begin{thm}[\cite{KS}]
Every probability measure preserving Boolean action of the isometry group of a separable locally compact metric space has a spatial realization.
\end{thm}

\medskip
{\bf (D)} The referee has suggested an alternative, more direct proof of \cref{T:1} that avoids the use of the result in \cite{KS} (whose proof depends on the solution to the Hilbert 5th Problem). We next sketch this argument. First note that if $(G_n)$ is a sequence of Polish groups and $G_n$ admits a free Borel action on a standard Borel space $X_n$ that preserves a probability measure $\mu_n$, then the product action of $\prod_n G_n$ on $\prod_n X_n$ is free and preserves $\prod_n \mu_n$. So by \cref{21} and the argument in the proof of \cref{23}, it is enough to show that any group of the form $S_\infty \ltimes K^\bbN$, where $K$ is a Polish locally compact group, admits a Borel action on a standard Borel space with invariant Borel probability measure for which the free part of the action has measure 1. For that use the result of Golodets and Sinel'shchikov \cite{GS} (see also \cite[Section 1]{AEG}) that $K$ has a free Borel action on a standard Borel space $X$ which admits an invariant probability Borel measure $\mu$. Let $\lambda$ be the Lebesgue measure on $[0,1]$ and define the action of $S_\infty \ltimes K^\bbN$ on $X^\bbN \times [0,1]^\bbN$ by 
\[
(\sigma, (g_n)) \cdot ((x_n), (y_n)) = ((g_n\cdot x_{\sigma^{-1}(n)}), (y_{\sigma^{-1} (n)})).
\]
This is a Borel action that preserves the product measure $\mu^\bbN \times \lambda^\bbN$ and its free part has measure 1.

\medskip
{\bf (E)} Finally we note that at the end of the paper we will give another proof of \cref{T:1} for non-archimedean Polish groups, using Bernoulli shifts and \cref{21}.

\section{Proof of \cref{T:2}}

{\bf (A)} Let $Z$ be a separable Banach space  viewed as a Polish group under addition.  We will use additive notation and we will
set $d$ to be the metric induced by the norm on $Z$: 
\[d(z,z')=||z-z'||.\]
Let $\mathrm{Lip}(Z)$ be the space of all 1-Lipschitz maps from $Z$ to $\mathbb{R}$. We endow $\mathrm{Lip}(Z)$ with the pointwise convergence topology and view it as a Polish $Z$-space by setting
 $(\zeta,f)\mapsto \zeta\cdot f$  with $(\zeta\cdot f)(z)=f(\zeta+z)$ for every $f\in \mathrm{Lip}(Z)$ and every $\zeta,z\in Z$. For every $z_1,\ldots,z_k\in Z$, every $\varepsilon>0$, and every  $f\in \mathrm{Lip}(Z)$ we will denote by $U_{z_1,\ldots,z_k,\varepsilon,f}$  the basic open subset $U$ of $\mathrm{Lip}(Z)$ defined by
 \[U=\big\{f' \in \mathrm{Lip}(Z) \mid |f(z_1)-f'(z_1)|<\varepsilon, \ldots, |f(z_k)-f'(z_k)|<\varepsilon\big\}.\]

In order to prove \cref{T:2} it suffices,  by \cite[Theorem 1.3]{Hj05}, to find a non-empty, $Z$-invariant, $G_{\delta}$ subset $X$ of $\mathrm{Lip}(Z)$ so that the $Z$-space $X$ is stormy. Recall that a Polish $G$-space $X$ is said to be {\bf stormy} \cite{Hj05}, if for every non-empty and open $V\subseteq G$ and every $x\in X$ we have that 
the map $g\mapsto g\cdot x$ from $V$ to $V x$ is not an open function. 

\begin{lem}\label{L:stormy}
Let $G$ be a Polish group acting continuously and freely on a Polish space $X$. Then the action is stormy if for every $x\in X$  and  every open $V\subseteq G$ with $1\in V$ there is an open $W\subseteq V$, $W\neq\emptyset,$ so that for all open $U\subseteq X$ with $x\in U$ there is $g\in V$ so that 
\begin{enumerate}
\item $gx\in U$;
\item $g\not\in W$.
\end{enumerate}
\end{lem}
\begin{proof}
By translating the arbitrary open set $V$ from the definition to the identity using the right translation it is easy to see that storminess is implied by the condition above, after one replaces (2) with 
\begin{enumerate}
\item[($2'$)] $gx\not\in Wx$.
\end{enumerate}
But $(2)$ is equivalent to $(2')$, given that the action is free.
\end{proof}

\begin{lem}\label{L:actsfreely}
There is a  $Z$-invariant, dense and $G_{\delta}$ subset $X_0$ of $\mathrm{Lip}(Z)$ so that  $Z$ acts freely on $X_0$.
\end{lem}
\begin{proof}
Let $\zeta$ be any element of $Z$ with $\zeta\neq 0$ and consider the set
\[X_{\zeta}=\{f\in \mathrm{Lip}(Z) \mid \; \exists z \in Z \; \; |f(\zeta+z)-f(z)|>||\zeta||/2   \}.\]
It is clear that $X_\zeta$ is open in $\mathrm{Lip}(Z)$. To see that $X_{\zeta}$ is also dense in $\mathrm{Lip}(Z)$ let $U:=U_{z_1,\ldots,z_k,\varepsilon,f}$ be a basic open subset of $\mathrm{Lip}(Z)$  and let $r>0$ be any number with $||z_1||<r,\ldots, ||z_k||<r$. Let also $[a,b]\subseteq \mathbb{R}$ be any interval of length at most $2r$ containing the values $f(z_1),\ldots,f(z_k)$. Set $z$ to be any positive multiple of $\zeta$ with $||z||>3r+2||\zeta||$. The  assignment 
\[z_1\mapsto f(z_1), \; \ldots, \; z_k\mapsto f(z_k), \;  z\mapsto b,  \; (z+\zeta)\mapsto b+3/4||\zeta|| \]
is clearly a 1-Lipschitz map from $\{z_1,\ldots,z_k, z,z+\zeta\}$ to $\mathbb{R}$ and therefore, by the classical theorem of McShane \cite{McS34}, it extends to some $f'\in \mathrm{Lip}(Z)$. Notice that any such $f'$  automatically lies in $X_{\zeta}\cap U$.

By a simple use of the triangle inequality it now follows that $\bigcap_{\zeta\in\mathcal{D}}X_{\zeta}$ is the desired set, where $\mathcal{D}$ is any countable dense subset of $Z\setminus\{0\}$.
\end{proof}

\begin{lem}\label{L:conditionGdelta}
There is a  $Z$-invariant, dense and $G_{\delta}$ subset $X_1$ of $\mathrm{Lip}(Z)$ so that for every open $V\subseteq Z$ which contains the identity and every $x\in X_1$ the condition stated in Lemma \ref{L:stormy} is satisfied.
\end{lem}
\begin{proof}
It suffices to check the condition in Lemma \ref{L:stormy} for $V$ ranging over some basis of open neighborhoods of $0$ in $Z$. So let $V=B(0,r)$, be an open ball  of radius $r>0$ around $0$ in $Z$. Choose $W$ to be the open ball    $B(0,r/2)$  of radius $r/2$ around $0$ in $Z$. Let also $z_1,\ldots,z_k$ be any finite list of points from $Z$ and let $\varepsilon>0$. Consider the set  $\mathcal{S}_{r,z_1,\ldots,z_k,\varepsilon}$  consisting of all $f\in\mathrm{Lip}(Z)$ for which there is a $\zeta \in V\setminus W$ with $\zeta \cdot f \in U_{z_1,\ldots,z_k,\varepsilon,f}$. In other words, $\mathcal{S}_{r,z_1,\ldots,z_k,\varepsilon}$ is simply the set
\[\bigg\{f\in \mathrm{Lip}(Z) \mid  \; \exists \zeta \in V\setminus W \text{ such that } |f(\zeta+z_i)-f(z_i)|<\varepsilon \text{ for all } i  \bigg\}\]

{\it Claim.} $\mathcal{S}_{r,z_1,\ldots,z_k,\varepsilon}$ is open and dense in $\mathrm{Lip}(Z)$. 
\begin{proof}[Proof of Claim.]\renewcommand{\qedsymbol}{\openbox\rlap{\textsubscript{Claim}}}
The set $\mathcal{S}_{r,z_1,\ldots,z_k,\varepsilon}$ is clearly open. To see that it is also dense, let  $U_{y_1,\ldots,y_l, \delta, f}$ be any basic open subset of $\mathrm{Lip}(Z)$. We will show that 
\[U_{y_1,\ldots,y_l, \delta, f}\cap\mathcal{S}_{r,z_1,\ldots,z_k,\varepsilon}\neq \emptyset.\]
 By shrinking both sets, if necessary, we can assume without loss of generality that $k=l, z_1=y_1, \ldots, z_k=y_l, \varepsilon=\delta$. 
Notice that it suffices to find some $\zeta \in V\setminus W$ so that for all $i,j\leq k$ we have that $d(\zeta+z_i,z_j)\geq d(z_i,z_j)$ since, if this is the case,  the assignment:  
\[z_1\mapsto f(z_1), \ldots , z_k\mapsto f(z_k), \; (\zeta+z_1)\mapsto f(z_1), \ldots , (\zeta+z_k)\mapsto f(z_k)\] 
is 1-Lipschitz and it therefore extends by \cite{McS34} to some $f'\in \mathrm{Lip}(Z)$. Any such extension witnesses that the intersection above is indeed non-empty.

It is easy to see that the property which remains to be checked follows from the fact that for every finite  $A\subseteq Z$ there is a $\zeta\in V\setminus W$ with $d(\zeta,a)\geq d (0,a)$ for all $a\in A$. 
So let $a_1,\ldots,a_n$ be an enumeration of $A$ and for each $i\leq n$ let $a^{*}_i$ be a norming functional for $a_i$, i.e., an element of the dual of $Z$ with $||a^{*}_i||=1$ and $a^*_i(a_i)=||a_i||$. Let also $P_i=\{\eta\in Z\mid a^*_{i}(\eta)=0\}$. Since $||a^{*}_i||=1$, 
for every $\eta\in P_i$ we have that $d(\eta,a_i)\geq a^{*}_i(a_i-\eta)=a^{*}_i(a_i)= d(0,a_i)$. Since  $P=\bigcap_{i\leq n} P_i$ is a subspace of finite codimension and $Z$ is infinite dimensional, $P$ is unbounded. Taking any $\zeta\in P$ with $||\zeta||=3/4r$ finishes the proof.
\end{proof}
To finish the proof of the Lemma we can simply take $X_1$ to be the intersection of a countable collection of sets of the form $\mathcal{S}_{r,z_1,\ldots,z_k,\varepsilon}$, where $\{z_1,\ldots,z_k\}$ varies over all finite subsets of some fixed countable dense subset of $Z$, and $\varepsilon, r$ vary over a sequence of positive reals converging to $0$.
\end{proof}

We can now finish the proof of Theorem \ref{T:2}.
\begin{proof}[Proof of Theorem \ref{T:2}]
Let $X=X_0\cap X_1$ where $X_0$ and $X_1$ are the sets 
 described in Lemma \ref{L:actsfreely} and  Lemma \ref{L:conditionGdelta}, respectively. By Lemma \ref{L:stormy} the action of $Z$ on $X$ is stormy and therefore, by \cite{Hj05} the orbit equivalence relation $E^{Z}_{X}$ is not essentially countable. Finally the $E^Z_X$ is Borel since the action of $Z$ on $X$ is free.
\end{proof}

We record the following question, whose positive answer would  generalize Theorem \ref{T:2} to the class of all non-locally compact abelian Polish groups. 
\begin{question}
Let $G$ be a non-locally compact abelian Polish group. Does $G$ admit a two-sided invariant metric $d$ with the property: for every open $V\subseteq G$ with $1_G\in V$ there is an open $W\subseteq V$ with $1_G\in W$ so that for every finite $A\subseteq G$ there is $g\in V\setminus W$ with $d(g,a)\geq d(1_G,a)$, for every $a\in A$.
\end{question}

\medskip
{\bf (B)} The referee has suggested an alternative  proof of \cref{T:2} which relies on Theorem \ref{21}. We sketch here this argument. Let $D=\{d_n\mid n\in\mathbb{N}\}$ be a countable dense subset of  the Banach space $Z$ and let $D^{*}=\{d_n^* \mid
 n \in\mathbb{N}\}\subseteq Z^*$ be the collection of the associated norming functionals. Notice that the continuous group homomorphism $(Z,+)\to (\mathbb{R}^{\mathbb{N}},+)$,  given by $z\mapsto \big(d^*_n(z)\big)_n$, is injective, since $D$ is dense in $Z$. 
But $\mathbb{R}^{\mathbb{N}}$ is a subgroup of the compact group $K=(T^2)^{\mathbb{N}}$, where $T^2=S^1\times S^1$ is the $2$-dimensional torus. It follows that $\mathbb{R}^{\mathbb{N}}$, and therefore $Z$,  admit a free and probability measure preserving action by left translation on $K$---the latter is endowed with its normalized Haar measure. By Theorem \ref{21} we may now conclude Theorem \ref{T:2}.

\section{A game theoretic criterion for non-essential countability}\label{S:Game}

Let $X$ be a Polish $G$-space and let $x,y\in X$. Let also $V$ be an open neighborhood of $1$ in $G$.
We say that {\bf $y$ admits a $V$-approximation from $x$} if there is $g_{*}\in G$, and a sequence $(g_n)$ in $V$ so that $(g_ng_{*}x)$ converges to $y$ when $n\to\infty$. We denote this by  
\[x \succeq_V y.\]

The following criterion is the main result of this section.
 
\begin{thm}\label{Criterion}
Let $G$ be a Polish group and let  $X$ be a Polish $G$-space. Assume that for every open neighborhood of the identity $V$ of $G$, for every  $G$-invariant comeager subset $C$ of $X$ and for every non-meager subset $O$ of $X$ there are $c\in C$ and $o\in O$ so that $c \succeq_V o$ but $[c] \neq [o]$. Then $E^G_X$ is not essentially countable.
 \end{thm}

Notice that if $G$ is locally compact then one can always find a small enough $V$ as above so that $x\succeq_V y$ implies $[x]=[y]$. Moreover, it is well known that countable equivalence relations are induced by actions of countable discrete (and so locally compact) groups. Thus, Theorem \ref{Criterion} is an immediate consequence of the following result.
  
\begin{thm}
	\label{th:Homo}
	Suppose that $G$, $H$ are Polish groups, $X$ is a Polish $G$-space, and $Y$ is a Polish $H$-space. Let
	$f : X \rightarrow Y$ be a Baire-measurable $(E_X,E_Y)$-homomorphism. Then for every open neighborhood of the identity $W$ in $H$ there exist an open neighborhood of the identity $V$ in $G$, a $G$-invariant comeager $C \subseteq X$, and a non-meager $O \subseteq X$ such that $c \succeq_{V} o$ implies $f(c) \succeq_{W} f(o)$ for $c \in C$, $o \in O$. If $W=H$, one can put $V=G$, and $O$ can be chosen to be comeager.
\end{thm}

In order to prove Theorem \ref{th:Homo}, we define a variant of a game introduced in \cite{LP18}, which we denote by $\Appr_{G,V}(x,y)$. For a Polish $G$-space $X$, an open neighborhood of the identity $V$ in $G$, and $x, y \in X$, Odd and Eve play as follows:

\begin{enumerate}
	\item in the first turn, Odd plays an open neighborhood $U_1$ of $y$, and Eve responds with an element $g_* \in G$;
	\item in the $n$-th turn, $n>1$, Odd plays an open neighborhood $U_n$ of $y$, and Eve responds with an element $g_{n-1} \in V$.
\end{enumerate}

The players proceed in this way, producing an element $g_* \in G$, a sequence $\{g_n\}$ of elements of $V$, and a sequence $\{U_n\}$ of open neighborhoods of $y$ in $X$. Eve wins the game if $g_*x \in U_1$, and $g_{n-1}g_*x \in U_n$ for $n>1$. Clearly, Eve has a winning strategy in the game $\Appr_{G,V}(x,y)$ iff $x \succeq_V y$.

It is straightforward to observe the following fact:

\begin{lemma}
	\label{le:ComSuffices}
	If Eve has a winning strategy in the game $\Appr_{G,V}(x,y)$, then she also has a winning strategy in the same game but with the additional winning conditions that each $g_n$, $n>0$, belongs to some given comeager subset $V^*$ of $V$, and each $g_{n-1}g_*x$, $n>1$, belongs to some given comeager subset $C$ of $X$, provided that the set of $g \in G$ such that $gx \in C$ is comeager.
\end{lemma}

In \cite{LP18}, the following well known fact is stated and proved (Lemma 2.5):

\begin{lemma}
	\label{le:Continuous}
	Suppose that $G$, $H$ are Polish groups, $X$ is a Polish $G$-space, and $Y$ is a Polish $H$-space. Let
	$f: X \rightarrow Y$ be a Baire-measurable $(E_X,E_Y)$-homomorphism. Then there exists a dense $G_\delta$ subset $C \subseteq X$ such
	that
	
	\begin{itemize}
		\item the restriction of $f$ to $C$ is continuous;
		\item for any $x \in C$, $\{g \in G : gx \in C \}$ is a comeager subset of $G$; 
		\item for any open neighborhood $W$ of the identity in $H$, and $x_1 \in C$ there exists an open neighborhood
		$U$ of $x_1$ and an open neighborhood $V$ of the identity in $G$ such that for any $x \in U \cap C$, and for a comeager
		set of $g \in V$, we have that $gx \in C$, and $f(gx) \in W f(x)$.
	\end{itemize}
	
\end{lemma}

\begin{proof}[Proof of Theorem \ref{th:Homo}]
	Fix neighborhood bases $\{V_n\}$, $\{W_n\}$ at the identity in $G$, $H$, respectively. Without loss of generality, we can assume that $W=W_{n_0}$ for some $n_0$. Define $C$ as the intersection of a set obtained from Lemma \ref{le:Continuous}, and the comeager set $\{x \in X: \forall^*g\in G \ gx \in C \}$. It is clearly $G$-invariant. To define $O$, let $N$ be the function assigning to a given $c \in C$, the smallest $n$ such that $V=V_{n}$ is as given by Lemma \ref{le:Continuous} for $c$ and $W$, and $\infty$ if there is no such $n$. In the proof of \cite[Lemma 2.5]{LP18}, it is showed that $N$ is analytic and takes finite values on a comeager subset of $X$. Thus, $N$ takes a constant value $n_1 \in \NN$ on some non-meager $O' \subseteq X$. Put $O=O' \cap C$, and $V=V_{n_1}$.
	
	Suppose now that Odd and Eve play the game $\Appr_{H,W}(f(c),f(o))$, where $c \in C$, $o \in O$, and $c \succeq_{V} o$. After Odd plays $U_1 \subseteq Y$ in the first turn, let Eve play $g_* \in G$ in the first turn of the game $\Appr_{G,V}(c,o)$, assuming that Odd chose $U \cap U'_1 \ni o$, where $U \subseteq X$ is given by Lemma \ref{le:Continuous}, and $U'_1 \subseteq X$ is chosen so that $f[U'_1 \cap C] \subseteq U_1$ (which is possible because $f$ is continuous on $C$). Clearly, she can choose $g_*$ so that $g_*x \in U \cap U'_1 \cap C$. In the next turns, for $U_n \subseteq Y$ played by Odd in $\Appr_{H,W}(f(c),f(o))$, we let Odd play $U \cap U'_n$ in $\Appr_{G,V}(c,o)$, where each $U'_n \subseteq X$ is such that $f[U'_n \cap C] \subseteq U_n$. Then, using Lemma \ref{le:ComSuffices}, Eve plays $g_{n-1} \in V$ according to her winning strategy, and so that $g_{n-1}g_*c \in X_0$, and, in the game $\Appr_{H,W}(f(c),f(o))$, Eve can choose $h_* \in H$, $h_{n-1} \in W$ such that $h_{n-1}h_*f(c)=f(g_{n-1}g_*c) \in U_n$. This procedure shows that Eve has a winning strategy in the game $\Appr_{H,W}(f(c),f(o))$, i.e., $f(c) \succeq_{W} f(o)$.
	
	The last statement is obvious.
\end{proof}

\section{A Bernoulli shift action for non-archimedean Polish groups}\label{S:Bernoulli}

Let $G$ be a non-archimedean Polish group, i.e., a Polish group admitting a neighborhood basis at the identity of open subgroups. Then there is a countable language $ \mathcal{L} $ and an $\mathcal{L}$-structure $\mathcal{N}$ with universe $\NN$ so that
\[ G=\mathrm{Aut}(\mathcal{N}) \subseteq {\rm Sym}(\NN) \]
and the basis can be taken to be the pointwise stabilizers of finite subsets of $ \NN$. 
Let us refer to such an action $G\curvearrowright\NN$ as a {\bf fundamental action}. Topological properties of $G$ correlate to properties of these actions. For example:  $G$ is compact if and only if all orbits of this action are finite; $G$ is locally compact if and only if there is an open subgroup $V$ of $G$ so that every orbit in the inherited action $ V \curvearrowright \NN $ is finite; and $G$ is CLI, i.e., it admits a complete and left invariant metric, if and only if for every sequence $(g_k)$ in $G$, which is Cauchy with respect to some left invariant metric, we have that every $n\in\mathbb{N}$ can be attained as the limit $(g_k\cdot m\to_{k\to\infty} n)$ for some $m\in\mathbb{N}$, see  \cite{Gao'}.

In this section we show that many topological properties of $G$, similarly translate to \emph{generic dynamical properties} of another Polish $G$-space which we call the  \emph{Bernoulli shift for $G$}. These generic dynamics have  immediate consequences for the complexity of the associated orbit equivalence relation. The {\bf Bernoulli shift for $G$} is the action of $G$ on the space $\mathbb{R}^\NN$ of all maps $x=(x(0),x(1),\ldots)$ from $\NN$ to $\RR$, obtained by permuting the coordinates, i.e.,  $(g\cdot x)(n):=x(g^{-1}(n))$.  If for any two sets $ X $ and $ Y $ we let $[X]^{Y}$ denote the set of injections from $ Y $ to $ X $, we observe that $[\mathbb{R}]^{\NN}$ is a $ G $-invariant, dense, $G_{\delta}$ subset of $\mathbb{R}^\NN$.

For our first application, recall (see e.g., \cite[Chapter 6]{Gao}) that a $ G $-space is \emph{generically ergodic} if every $G$-invariant Borel set is meager or comeager. Generic ergodicity is equivalent to the existence of a dense orbit. If moreover, every orbit is meager  then the corresponding orbit equivalence relation must be \emph{non-smooth}. In Application 1, we see that the weaker non-compactness property of a fundamental action $ G \curvearrowright \NN $ (the existence of an infinite orbit) is upgraded to a stronger non-compactness property (generic ergodicity) in the case of the Bernoulli shift.

\subsection*{Application 1.} $G$ is not compact if and only if the $G$-space $\mathbb{R}^{\NN}$ has a $ G $-invariant subspace that is generically ergodic with meager orbits.
\begin{proof}
The right-to-left direction follows immediately from the aforementioned results that generic ergodicity together with meager orbits imply non-smoothness, recalling the fact mentioned in the introduction that every action of a compact group has a smooth orbit equivalence relation. To prove the converse, recall Neumann's lemma that for any finite $A \subseteq \NN$ of points with infinite orbits, there exists $g \in G$ such that $A \cap g[A]=\emptyset$. If $G$ is not compact, there exists an infinite orbit $O$ in $\NN$. Let $ Y = \{z \in \mathbb{R}^{\NN} \mid z(a) = 0 \text{ for every } a \in \NN \setminus O\} $. Then $ Y $ is uncountable and $ G $-invariant. Now fix any $ \hat{a} \in O $. Then for every $ r \in \mathbb{R} $, the set $ \{z \in Y \mid z(\hat{a}) = r \} $ is closed, nowhere dense, and every orbit of the action $ G \curvearrowright Y $ is contained in a countable union of such sets. Finally, using Neumann's lemma one easily constructs an element of $ Y $ with dense orbit.
\end{proof}
\begin{rem}
While we specified a subspace with meager orbits and constructed a point in it with dense orbit, it is an application of Lemma \ref{le:GoodSeq} below that in fact, when $ G $ is not compact, the orbit closure $ \overline{G \cdot z} $ of the generic $ z \in \mathbb{R}^{\mathbb{N}} $ has the stated properties.
\end{rem}

For the next application we first recall some definitions and a theorem from \cite{LP18}. Let $H$ be any Polish group. A {\bf left-Cauchy} sequence is any sequence $(h_n)$ in $H$ which is Cauchy with respect to some topologically compatible left-invariant metric on $H$. It is {\bf right-Cauchy} if $(h^{-1}_n)$ is left-Cauchy. If $X$ is a Polish $H$-space and $x,y$ are elements of $X$, then we say that $x$ {\bf right Becker-embeds into $y$} if there is a right Cauchy sequence $(h_n)$ in $H$ with $h_n y\to x$. The main theorem from \cite{LP18} states that: if for every comeager subset $C$ of $X$ there are $x,y\in C$ so that $[x]\neq [y]$ and $x$ right Becker-embeds into $y$, then $E^H_X$ is not Borel reducible to any orbit equivalence relation that is induced by the continuous action of a CLI Polish group. In Application 2, we show that the dynamical property from \cite{Gao'} which characterizes when $G$ is CLI in terms of the action $G\curvearrowright \NN$,
upgrades in the Bernoulli shift of $G$ to the generic dynamical behavior from \cite{LP18} which we just described. As a corollary, we get that every non-archimedean Polish group which is not CLI, admits a continuous action on a Polish space whose associated orbit equivalence relation is not classifiable by actions of CLI groups. Notice that this corrollary also follows from \cite{Tho}, since CLI orbit equivalence relations are pinned.

\subsection*{Application 2.} $G$ is not CLI if and only if the $G$-space $\mathbb{R}^{\NN}$ satisfies the criterion from \cite{LP18}. 

\begin{proof}
The implication from right to left is clear. To prove the converse, denote by $d$ be the left-invariant metric on $[\NN]^{\NN}$ defined by $$d(x,y)=\max\{2^{-n}:x(n) \neq y(n)\}$$ for $x \neq y$. As $G$ is non-CLI, and so $d$ is not complete on $G$ (see \cite[Lemma 2.1]{Gao}), we can fix a non-convergent Cauchy sequence $\{h_n\}$ with respect to $d$. Because $\{h_n\}$ clearly converges in $[\NN]^{\NN}$ (the definition of $d$ warranties that, for every $m$, the value $h_n(m)$ is fixed starting from some $n$ on), and $G$ is closed in ${\rm Sym}(\NN) \subseteq [\NN]^{\NN}$, $\{h_n\}$ actually converges to some $h \in [\NN]^{\NN} \setminus {\rm Sym}(\NN)$, i.e., $h[\NN] \subsetneq \NN$. Let $\phi:[\mathbb{R}]^{\NN} \rightarrow [\mathbb{R}]^{\NN}$ be defined by $$\phi(x)(n)=x(h(n)),$$ $n \in \NN$. Clearly, $\phi$ is an open surjection.
	
Fix a comeager $C \subseteq [\mathbb{R}]^{\NN}$. Then $\phi[C]$ is also comeager, and so there exists $c \in C$ such that $o=\phi(c) \in C$. As $o[\NN] \subsetneq c[\NN]$, $[c] \neq [o]$, however $h^{-1}_n c \rightarrow o$. Using the fact that $\{h^{-1}_n\}$ is a right Cauchy sequence in $X$, i.e., it is Cauchy for the right-invariant metric $d'$ defined by $d'(x,y)=d(x^{-1},y^{-1})$, it is straightforward to verify that $o$  right-Becker embeds into $o$. 
\end{proof}

In Application 3, we similarly show that the failure of local compactness of $G$---i.e. the existence of infinite orbits for all inherited actions $ V \curvearrowright \NN $ by open subgroups---is amplified to the property of Theorem \ref{Criterion} in the Bernoulli shift.

\subsection*{Application 3.} $G$ is not locally compact iff the $G$-space $\mathbb{R}^{\NN}$ satisfies the criterion from Theorem \ref{Criterion}. 

We start with a strengthening of Neumann's lemma.

\begin{lemma}
	\label{le:GoodSeq}
	Let $ G \curvearrowright \NN $ be a fundamental action and $\NN = A \sqcup B $ a partition where $ B $ consists of those elements with infinite $ G $-orbits and $ A $ the elements with finite orbits. Then there exists a sequence $\{g_n\} \subseteq G$ such that:
	\begin{enumerate}
		\item $F \cap g_n[F]=\emptyset$ for every finite $F \subseteq B$, and almost all $n$;
		\item $g_n(a)=a$ for every $a \in A$, and almost all $n$.
	\end{enumerate}
\end{lemma}

\begin{proof}
	Fix finite $F \subseteq B$, $D \subseteq A$. Put $F_0=F$, and, using Neumann's lemma, fix some $h_0 \in G$ such that $F_0 \cap h_0[F_0]=\emptyset$. Put $F_1= F_0 \cup h^{-1}_0[F_0]$, and fix $h_1 \in G$ such that $F_1 \cap h_1[F_1]=\emptyset$. In this way, construct $F_n \subseteq B$, $h_n \in G$, $n \in \NN$. As each element of $D$ is in a finite orbit, there must be $m<n$ such that $h^{-1}_m  \upharpoonright D= h^{-1}_n  \upharpoonright D$. But then $h_nh^{-1}_m \upharpoonright D$ is the identity, and, since $h^{-1}_m[F] \subseteq F_n$, we have that $F \cap h_nh^{-1}_m[F]=\emptyset$. Put $g_{F,D}=h_nh^{-1}_m$.
	
	Now, write $A$ and $B$ as increasing unions of finite sets $\bigcup A_n$ and $\bigcup_n B_n$, respectively, and construct $g_n=g_{B_n,A_n}$ as above. Clearly, $\{g_n\}$ is as required.
\end{proof}

\begin{proof}[Proof of Application 3]
	The implication from right to left is immediate by Theorem \ref{Criterion} and \cite{K92}. For the left-to-right implication, suppose that $ G $ is not locally compact and fix an open neighborhood of the identity, $ V $. As $ G $ is non-archimedean, $ V $ contains an open subgroup and so we may assume, without loss of generality, that $ V $ is a subgroup itself.
	
	Since $ G $ is not locally compact, $ V $ is not compact, and so some element in the fundamental action $ G \curvearrowright \NN $ has an infinite orbit in the induced action $ V \curvearrowright \NN $. Partition $ \NN = A \sqcup B $ as in Lemma \ref{le:GoodSeq} according to the action of $ V $, and note that the above says that $ B \neq \emptyset $, and so in fact, $ B $ must be infinite.
	
	Let $ \{b_{k}\}_{k \in \NN} $ be an enumeration of $ B $ and let $ \phi \colon [\RR]^{\NN} \to [\RR]^{\NN} $ be the map that fixes the coordinates coming from $ A $ and acts as a left-shift on the coordinates from $ B $, according to the fixed enumeration. In other words, $ \phi $ is defined by
	\begin{align*}
	\phi(x)(b_{k}) &= x(b_{k+1}), \text{ when }  b_{k} \in B\\
	\phi(x)(a) &= x(a), \text{ when } a \in A
	\end{align*}
	
	Note the following about the map $ \phi $. First, that it is a continuous and open surjection of $ [\RR]^{\NN} $ onto itself and second, since points in $ [\RR]^{\NN} $ are \emph{injections} $ \NN \to \RR $ (i.e., non-repetitive sequences), the $ G $-orbits $ [z] $ and $ [\phi(z)] $ are not equal for any $ z \in [\RR]^{\NN} $.
	
	Now, apply Lemma \ref{le:GoodSeq}, to the action $ V \curvearrowright \NN $ to obtain a sequence $ \{g_{n}\}_{n \in \NN} $ from $ V $ with properties (1) and (2) as in the lemma. Let $ C' $ consist of those points $ x \in [\RR]^{\NN} $ so that for any $ N \in \NN $, $ b_{k_{1}}, \dots, b_{k_{m}} \in B $, and open, rational intervals $ I_{1}, \dots, I_{m} \subseteq \RR $, there is an $ n \geqslant N $ with $ g_{n}(x)(b_{k_{j}}) \in I_{j} $ for all $ 1 \leqslant j \leqslant m $. Then by the properties of the sequence $ \{g_{n}\} $, $ C' $ is comeagre in $ [\RR]^{\NN} $ and for every $ x \in C' $ and $ y \in [\RR]^{\NN} $ satisfying $ x(a) = y(a) $ for all $ a \in A $, there is a subsequence $ \{g_{n_{l}}\}_{l \in \NN} $ with $ g_{n_{l}}x \to y $. In particular, this holds for any pair $ z $ and $ \phi(z) $ where $ z \in C' $.
	
	To finish verifying the conditions of Theorem \ref{Criterion}, let $ C $ be any comeagre invariant subset of $ [\RR]^{\NN} $ and let $ O $ be nonmeagre. Then choose any $ c \in (C \cap C' \cap \phi^{-1}[O]) $ and let $ o = \phi(c) $. Then $ c \in C $, $ o \in O $, and by the argument in the preceding paragraphs, $c \succeq_V o$ while $[c] \neq [o]$.
	\end{proof}

As a consequence of Application 3, every non-archimedean Polish group that is not locally compact admits a continuous action on its Bernoulli shift with an orbit equivalence relation that is not essentially countable. The Bernoulli shift  for $G$ can be also used to provide an alternative measure-theoretic proof of this fact, based on \cite[Theorem 1.6]{K92}. To see this, let $\mu$ be the product measure on $\mathbb{R}^{\NN}$ of any non-atomic probability measure on $\mathbb{R}$ and notice that it concentrates on  $[\mathbb{R}]^{\NN}$, where $G$ acts freely.


\begin{thebibliography}{MMM}

\bibitem[AEG94]{AEG} S. Adams, G.A. Elliott and T. Giordano, Amenable actions of groups, {\it Trans. Amer. Math. Soc.}, {\bf 344(2)} (1994), 803--822.





\bibitem[BK96]{BK} H. Becker and A.S. Kechris, {\it The descriptive set theory of Polish group actions}, Cambridge Univ. Press, 1996.



\bibitem[FM77]{Feldman} J. Feldman and C.C. Moore, Ergodic equivalence relations, cohomology, and von Neumann algebras, 
{\it Trans. Amer. Math. Soc.}, {\bf 234(2)} (1977), 289--324.

\bibitem[FR85]{FR} J. Feldman and A. Ramsey, Countable sections for free actions of groups, {\it Adv. Math}, {\bf 55} (1985), 224-227.


\bibitem[GK03]{GK} S. Gao and A.S. Kechris,  On the classification of Polish metric spaces up to isometry, {\it Mem. Amer. Math. Soc.}, {\bf 766} (2003).

\bibitem[GTW05]{GTW} E. Glasner. B. Tsirelson and B. Weiss, The automorphism group of the Gaussian measure cannot act pointwise, {\it Israel J. Math.}, {\bf 148} (2005), 305--329.


\bibitem[GW05]{GW} E. Glasner and B. Weiss, Spatial and non-spatial actions of Polish groups, {\it Erg. Th. \& Dynam. Sys.}, {\bf 25} (2005), 1521--1538.

\bibitem[GS90]{GS} V.Ya. Golodets and S.D. Sinel'shchikov, Amenable ergodic group actions and images of
cocycles, {\it Dokl. Akad. Nauk SSSR}, {\bf 312} (1990), 1296-1299; transi, in {\it Soviet Math. Dokl.}, {\bf 41}
(1990), 523-526.



\bibitem[Hj00]{Hj00} G. Hjorth,  Classification and orbit equivalence relations, {\it Amer. Math. Soc.}, {\bf 75} (2000).


\bibitem[Hj05]{Hj05} G. Hjorth,  {\it A dichotomy for being essentially countable},  Contemp. Math, {\bf 380(7)} (2005), 109--127.


\bibitem[Ke92]{K92} A.S. Kechris, Countable sections for locally compact group actions, {\it Erg. Th. \& Dynam. Sys.}, {\bf 12} (1992), 283--295.


\bibitem[Ke95]{K95} A.S. Kechris, {\it Classical Descriptive Set Theory}, Graduate Texts in Mathematics, Springer-Verlag, 
1995.


  \bibitem[Ke10]{K11} A.S. Kechris, {\it Global aspects of ergodic group actions}, Amer. Math. Soc., 2010.
 
  \bibitem[Ke19]{K19} A.S. Kechris, The theory of countable Borel equivalence relations,  {\it preprint,  www.math.caltech.edu/$\sim$kechris/} (2020). 

  
  \bibitem[KS11]{KS} A. Kwiatkowska and S. Solecki, Spatial models of Boolean actions and groups of isometries, {\it Erg. Th. \& Dynam. Sys.} {\bf 31(2)} (2011), 405--421.  
  


  \bibitem[LP18]{LP18} M. Lupini and A. Panagiotopoulos, Games orbits play and obstructions to Borel reducibility, {\it Groups, Geometry, and Dynamics} {\bf 12(4)} (2018), 1461--1483.  



  \bibitem[McS34]{McS34} E. J. McShane, Extension of range of functions, {\it Bull. Amer. Math. Soc.}, {\bf 40} (1934), 837--842.  
  
  
  \bibitem[Ma16]{Ma} M. Malicki, Abelian pro-countable groups and orbit equivalence relations, {\it Fund. Math.}, {\bf 233(1)} (2016), 83--99.  
  
  \bibitem[So00]{So} S. Solecki, Actions of non-compact and non-locally compact groups, {\it J. Symb. Logic}, {\bf 65(4)} (2000), 1881--1894.  

  \bibitem[Gao98]{Gao'} S. Gao, On Automorphism Groups of Countable Structures, {\it  J. Symb. Logic}, {\bf 63} (1998), 891--896.
     
  \bibitem[Ga08]{Gao} S. Gao, Invariant Descriptive Set Theory, Chapman and Hall, 2008.
  
  \bibitem[Tho06]{Tho} A. Thompson, A metamathematical condition equivalent to the existence of a complete left invariant metric,  {\it J. Symb. Logic}, {\bf 71(4)} (2006), 1108--1124.  
 
 
\end{thebibliography}
\end{document}